\documentclass{article}
\usepackage[left=20truemm, right=20truemm]{geometry}
\usepackage{amsmath,amssymb,amsthm}
\usepackage{bm}
\usepackage{mathtools}
\usepackage[shortlabels]{enumitem}
\usepackage{float}
\usepackage{cite}

\usepackage{xcolor,graphicx}
\usepackage{tikz}
\usetikzlibrary{calc}
\usepackage{subcaption}

\usepackage[pdfauthor={derajan},pdfstartview=XYZ,bookmarks=true,
colorlinks=true,linkcolor=blue,urlcolor=magenta,citecolor=blue,pdftex,bookmarks=true,linktocpage=true,hyperindex=true]{hyperref}

\theoremstyle{definition}
\newtheorem{theorem}{Theorem}[]
\newtheorem{proposition}[theorem]{Proposition}

\newtheorem{claim}[]{Claim}

\newtheorem{observation}[theorem]{Observation}

\newtheorem{question}[theorem]{Question}

\numberwithin{equation}{section}

\title{A note on the number of non-cycle components in a pseudo 2-factor of graphs}
\author{Masaki Kashima\thanks{Keio University, Yokohama, Japan. email: masaki.kashima10@gmail.com}}

\begin{document}

\maketitle

\begin{abstract}
    A pseudo 2-factor of a graph is a spanning subgraph such that each component is $K_1$, $K_2$, or a cycle. This notion was introduced by Bekkai and Kouider in 2009, where they showed that every graph $G$ has a pseudo 2-factor with at most $\alpha(G)-\delta(G)+1$ components that are not cycles. For a graph $G$ and a set of vertices $S$, let $\delta_G(S)$ denote the minimum degree of vertices in $S$.
    In this note, we show that every graph $G$ has a pseudo 2-factor with at most $f(G)$ components that are not cycles, where $f(G)$ is the maximum value of $|I|-\delta_G(I)+1$ among all independent sets $I$ of $G$.
    This result is a common generalization of a result by Bekkai and Kouider and a previous result by the author on the existence of a 2-factor.
\end{abstract}

\section{Introduction}\label{section:intro}

Throughout the paper, we only consider simple, finite, and undirected graphs.
For a graph $G$, let $\delta(G)$ and $\alpha(G)$ denote the minimum degree and the independence number, respectively.
For a graph $G$ and a set $S\subseteq V(G)$, let $N_G(S)$ denote the set of vertices in $V(G)\setminus S$ that have neighbors in $S$.
In particular, for a subgraph $H$ of $G$, we abbreviate $N_G(V(H))$ to $N_G(H)$.
For a positive integer $n$, let $K_n$ denote the complete graph of order $n$.

A \emph{2-factor} of a graph $G$ is a 2-regular spanning subgraph of $G$.
A \emph{Hamilton cycle} of a graph $G$, which is a cycle that passes through all vertices of $G$, is exactly a connected 2-factor.
Thus, sufficient conditions for a graph to have a 2-factor have been actively studied in connection with Hamilton cycles.

As a relaxation of a 2-factor, Bekkai and Kouider~\cite{BK2009} introduced a notion of pseudo 2-factor. The term ``pseudo 2-factor" was coined by them, though the concept had already been studied by Enomoto and Li~\cite{EL2004} in 2004.
A \emph{pseudo 2-factor} of a graph $G$ is a spanning subgraph of $G$ in which each component is isomorphic to $K_1$, $K_2$, or a cycle.
By allowing $K_1$ and $K_2$ as components, it is clear that every graph has a pseudo 2-factor.
Thus sufficient and/or necessary conditions for a graph to have a ``special" pseudo 2-factor have been studied in the literature.

Well before the term pseudo 2-factor established, Tutte~\cite{T1953} gave a sufficient and necessary condition for a graph to have a pseudo 2-factor without isolated vertices. 
Later, by Cornu\'{e}jols and Hartvigsen~\cite{CH1986}, the result was extended to a sufficient and necessary condition for a graph to have a pseudo 2-factor without isolated vertices and small odd cycles.
In 2018, Egawa and Furuya~\cite{EF2018} gave sufficient conditions, which are more easily checkable, for a graph to have a pseudo 2-factor with no isolated vertices and small odd cycles.
From the other aspect, motivated by a result on 2-factor with prescribed number of components, Enomoto and Li~\cite{EL2004} investigated the sufficient degree sum conditions for a graph to have a pseudo 2-factor with exactly $k$ components.
Recently, Chiba and Yoshida~\cite{CY2026} considered an analogue of the result for bipartite graphs.

In this note, we focus on the number of components that are isomorphic to $K_1$ or $K_2$ in a pseudo 2-factor.
A component of a pseudo 2-factor is called a \emph{non-cycle component} if it is isomorphic to $K_1$ or $K_2$.
Since a pseudo 2-factor without non-cycle components is a 2-factor of a graph, we are interested in upper bounds of the number of non-cycle components in a pseudo 2-factor of a given graph.
Bekkai and Kouider~\cite{BK2009} gave the following upper bound.

\begin{theorem}[\cite{BK2009}]\label{thm:original}
    For any graph $G$ with $\alpha(G)\geq \delta(G)$, $G$ has a pseudo 2-factor with at most $\alpha(G)-\delta(G)+1$ non-cycle components.
\end{theorem}

The bound in Theorem~\ref{thm:original} is best possible.
Indeed, for an arbitrary graph $H$ and a positive integer $p\geq |V(H)|+1$, let us consider the graph $G_1$ obtained from $H$ by joining $p$ disjoint copies of $K_2$.
Then it follows that $\delta(G_1)=|V(H)|+1$ and $\alpha(G_1)=p\geq |V(H)|+1$, both of which are satisfied by vertices in copies of $K_2$.
On the other hand, it is easy to see that every pseudo 2-factor of $G_1$ has at least $p-|V(H)|=\alpha(G_1)-\delta(G_1)+1$ non-cycle components since $G_1-V(H)$ consists of $p$ disjoint copies of $K_2$.

Their result with the case $\alpha(G)=\delta(G)$ implies the following theorem by Niessen~\cite{N1995}.

\begin{theorem}[\cite{N1995}]\label{thm:2factor niessen}
    Every graph $G$ with $\delta(G)\geq \alpha(G)+1$ has a 2-factor.
\end{theorem}

Recently, the author showed the following result, which extends Theorem~\ref{thm:2factor niessen} in a different way.
For a vertex set $S$ of a graph $G$, let $\delta_G(S)$ denote the minimum degree of the vertices in $S$.

\begin{theorem}[\cite{Karxiv}]\label{thm:2factor}
    If every independent set $I$ of $G$ satisfies $\delta_G(I)\geq |I|+1$, then $G$ has a 2-factor.
\end{theorem}

If a graph $G$ satisfies $\delta(G)\geq \alpha(G)+1$, then every independence set $I$ of $G$ satisfies 
\[
\delta_G(I)\geq \delta(G)\geq \alpha(G)+1\geq |I|+1,
\]
and hence Theorem~\ref{thm:2factor niessen} holds from Theorem~\ref{thm:2factor}.
In this note, we show the following result, which generalizes both Theorems~\ref{thm:original} and \ref{thm:2factor} (and obviously Theorem~\ref{thm:2factor niessen}). 
For a graph $G$, let 
\[
f(G):=\max\{|I|-\delta_G(I)+1\mid I\text{ is an independent set of }G\}.
\]
Then the following holds.

\begin{theorem}\label{thm:main}
    Every graph $G$ has a pseudo 2-factor with at most $\max\{0,f(G)\}$ non-cycle components.
\end{theorem}

We will give a proof of Theorem~\ref{thm:main} in the next section.

When every independent set $I$ of $G$ satisfies $\delta_G(I)\geq |I|+1$, then obviously $f(G)\leq 0$ and Theorem~\ref{thm:main} implies that $G$ has a 2-factor.
Also, for any graph $G$ and any independent set $I$ of $G$, we have $|I|-\delta_G(I)+1\leq \alpha(G)-\delta(G)+1$, and hence $f(G)\leq \alpha(G)-\delta(G)+1$.
Thus, Theorem~\ref{thm:main} implies both Theorems~\ref{thm:original} and \ref{thm:2factor}.
By the tightness of the bound in Theorem~\ref{thm:original}, the bound in Theorem~\ref{thm:main} is best possible as well.

We remark that the gap between $f(G)$ and $\alpha(G)-\delta(G)+1$ can be arbitrarily large.
For a positive integer $k$, we set two vertices $v_1,v_2$, two disjoint independent sets $A_1, A_2$ of order $k$, and a complete graph $B$ of order at least $2k$.
Let $G_2$ be a graph obtained by $v_1$, $v_2$, $A_1$, $A_2$, and $B$ by joining $A_i$ to $B\cup \{v_i\}$ for each $i\in \{1,2\}$ (Figure~\ref{fig:gap}).
Then it follows that $\delta(G_2)=d_G(v_1)=k$ and $\alpha(G_2)=|A_1\cup A_2|=2k$, implying that $\alpha(G_2)-\delta(G_2)+1=2k-k+1=k+1$.
On the other hand, since all the maximal independent sets of $G_2$ are $I_1=\{v_1,v_2,b\}$ with $b\in B$, $I_2=v_1\cup A_2$, $I_3=v_2\cup A_1$, and $I_4=A_1\cup A_2$, we have 
\[
f(G_2)=|I_2|-\delta_G(I_2)+1=(k+1)-k+1=2.
\]
Thus, the bound in Theorem~\ref{thm:main} is strictly smaller than that in Theorem~\ref{thm:original}.

\begin{figure}
    \centering
    \begin{tikzpicture}[roundnode/.style={circle, draw=black,fill=black, minimum size=1.5mm, inner sep=0pt}]
    \draw [rounded corners] 
    (-2.5,0.3)--(-0.3,0.3)--(-0.3,-0.3)--(-2.5,-0.3)--cycle;
    \draw [rounded corners] 
    (2.5,0.3)--(0.3,0.3)--(0.3,-0.3)--(2.5,-0.3)--cycle;
    \draw [rounded corners] 
    (-2.5,2.3)--(2.5,2.3)--(2.5,1.7)--(-2.5,1.7)--cycle;
    
    \node [roundnode] (v1) at (-1.4,-1){};
    \node [roundnode] (v2) at (1.4,-1){};

    \draw (v1)--(-2.2,-0.3);
    \draw (v1)--(-0.6,-0.3);
    \draw (v2)--(2.2,-0.3);
    \draw (v2)--(0.6,-0.3);
    \draw (-2.2,0.3)--(-2.2,1.7);
    \draw (-0.6,0.3)--(2.2,1.7);
    \draw (0.6,0.3)--(-2.2,1.7);
    \draw (2.2,0.3)--(2.2,1.7);

    \node at (-1,-1){$v_1$};
    \node at (1.8,-1){$v_2$};
    \node at (-2.8,0){$A_1$};
    \node at (2.8,0){$A_2$};
    \node at (2.8,2){$B$};
    \node at (-1.4,0){$\overline{K_k}$};
    \node at (1.4,0){$\overline{K_k}$};
    \node at (0,2){$K_\ell$ ($\ell\geq 2k$)};
    \end{tikzpicture}
    \caption{A graph $G_2$ which has a gap between $f(G_2)$ and $\alpha(G_2)-\delta(G_2)+1$.}
    \label{fig:gap}
\end{figure}
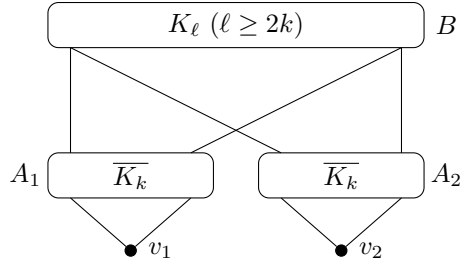

\section{Proofs}\label{section:proof}

We first show some statements used in our proof of Theorem~\ref{thm:main}.
The following observation can be easily verified.

\begin{observation}\label{obs:leaf independent set}
    For every tree $T$ and every leaf $u$ of $T$, $T$ has a maximum independent set that contains $u$.
\end{observation}

By using well-known K\"{o}nig's theorem on matchings and vertex covers of bipartite graphs, we can show the following.

\begin{proposition}\label{prop:tree}
    Every tree $T$ has a pseudo 2-factor with exactly $\alpha(T)$ components.
\end{proposition}

\begin{proof}
    The statement is trivial when $|V(T)|=1$.
    Suppose that $n:=|V(T)|\geq 2$.
    Let $\beta(T)$ be the minimum cardinality of a vertex cover of $T$.
    
    For every independent set $I$ of $T$,
    $V(T)\setminus I$ is a vertex cover of $T$ since every vertex in $I$ has at least one neighbor in $T$ that must be in $V(T)\setminus I$.
    Thus, we have $\beta(T)=n-\alpha(T)$.
    Since $T$ is a bipartite graph, by K\"{o}nig's theorem, $T$ has a matching $M$ with $\beta(T)=n-\alpha(T)$ edges.
    Combining all the edges in $M$ and the vertices in $V(G)-V(M)$, we obtain a pseudo 2-factor of $T$ with 
    \[
    |M|+(n-2|M|)=n-\alpha(T)+(n-2(n-\alpha(T)))=\alpha(T)
    \]
    components, as desired.
\end{proof}

By Observation~\ref{obs:leaf independent set}, every tree $T$ satisfies $f(T)\leq \alpha(T)-1+1=\alpha(T)$, and hence Proposition~\ref{prop:tree} states that Theorem~\ref{thm:main} holds for trees.
Furthermore, this directly implies the following.

\begin{proposition}\label{prop:forest}
    Every forest $G$ has a pseudo 2-factor with exactly $\alpha(G)$ components.
    In particular, every forest $G$ has a pseudo 2-factor with at most $f(G)$ components.
\end{proposition}

Now we prove Theorem~\ref{thm:main}.

\subsection{Proof of Theorem~\ref{thm:main}}\label{subsection:proof}

The case $G$ is a forest is done by Proposition~\ref{prop:forest}.
Suppose that $G$ has at least one cycle.
Let $F$ be a union of pairwise vertex-disjoint cycles of $G$ such that
\begin{enumerate}[label=(\alph*)]
    \item\label{cond:order} $|V(F)|$ is as large as possible, and
    \item\label{cond:isolated vertex} subject to (a), the number of isolated vertices in $G-V(F)$ is as small as possible.
\end{enumerate}
If $V(F)=V(G)$, then obviously $F$ is a 2-factor of $G$ and we are done. 
Thus, we assume that $V(G)\setminus V(F)\neq \emptyset$.
By the maximality of $F$, it follows that $H:=G-V(F)$ is a forest.
Set $\alpha:=\alpha(H)$.
By Proposition~\ref{prop:forest}, $H$ has a pseudo 2-factor $F_H$ of exactly $\alpha$ components.
Then $F\cup F_H$ is a pseudo 2-factor of $G$ with exactly $\alpha$ non-cycle components.
Thus, it suffices to show that $\alpha\leq f(G)$.

Assume to the contrary that $\alpha>f(G)$.
We set one orientation of each cycle $C$ of $F$.
For each vertex $v\in V(F)$, let $v^+$ denote the successor of $v$ and let $v^-$ denote the predecessor of $v$ along the orientation of the cycle of $F$ containing $v$.

\begin{claim}\label{claim:next vertex independent}
    For a vertex $x$ of $H$, if $x$ has two neighbors $y_1$ and $y_2$ in $V(F)$, then $y_1^+y_2^+\notin E(G)$.
\end{claim}

\begin{proof}
    Assume to the contrary that $y_1^+y_2^+\in E(G)$ for some $y_1, y_2\in N_G(x)\cap V(F)$.
    Then, $F'=F\cup \{y_1x,y_2x, y_1^+y_2^+\}-\{y_1y_1^+,y_2y_2^+\}$ is a 2-regular subgraph of $G$ such that $V(F')=V(F)\cup \{x\}$, a contradiction by the maximality of $F$.
\end{proof}

\begin{claim}\label{claim:degree 0 vertex}
    For every isolated vertex $x$ of $H$, there are two vertices $y, y'\in N_G(x)\cap V(F)$ such that $N_G(y^+)\cap V(H)\neq \emptyset$ and $N_G(y'^+)\cap V(H)\neq \emptyset$.
\end{claim}

\begin{proof}
    For an isolated vertex $x$ of $H$, we set $Y^+=\{y^+\mid y\in N_G(x)\cap V(F)\}$.
    Suppose, for the sake of contradiction, that $|Y^+\cap N_G(H)|\leq 1$.
    Let $I$ be a maximum independent set of $H$.
    Note that $x\in I$ since $x$ is an isolated vertex of $H$.
    By Claim~\ref{claim:next vertex independent}, $Y^+$ is an independent set of $G$, and hence $I':=I\cup (Y^+\setminus N_G(H))$ is an independent set of $G$.
    Since $|Y^+\setminus N_G(H)|\geq |Y^+|-1=d_G(x)-1$,
    \[f(G)\geq |I'|-\delta_G(I')+1\geq (|I|+|Y^+\setminus N_G(H)|)-d_G(x)+1\geq \alpha+d_G(x)-1-d_G(x)+1=\alpha,\]
    a contradiction.
\end{proof}

\begin{claim}\label{claim:degree 1 vertex}
    For every vertex $x$ of $H$ with $d_H(x)=1$, there is a vertex $y\in N_G(x)\cap V(F)$ such that $N_G(y^+)\cap V(H)\neq\emptyset$.
\end{claim}

\begin{proof}
    For a vertex $x$ of $H$ with $d_H(x)=1$, set $Y^+=\{y^+\mid y\in N_G(x)\cap V(F)\}$.
    Note that $|Y^+|=|N_G(x)\cap V(F)|=d_G(x)-1$.
    Assume, for the sake of contradiction, that $Y^+\cap N_G(H)=\emptyset$.
    By Observation~\ref{obs:leaf independent set}, $H$ has a maximum independent set $I$ that contains $x$.
    By Claim~\ref{claim:next vertex independent}, $Y^+$ is an independent set of $G$.
    This, together with the assumption that $Y^+\cap N_G(H)=\emptyset$ implies that $I':=I\cup Y^+$ is an independent set of $G$ that contains $x$, and hence
    \[f(G)\geq |I'|-\delta_G(I')+1\geq (|I|+|Y^+|)-d_G(x)+1=\alpha+d_G(x)-1-d_G(x)+1=\alpha,\]
    a contradiction.
\end{proof}

In the rest of the proof, using Claims~\ref{claim:degree 0 vertex} and \ref{claim:degree 1 vertex}, we shall construct a 2-regular subgraph of $G$ which contradicts the choice of $F$.

Let $D_0$ be a component of $H$ and choose $x_0\in V(D_0)$ with $d_{D_0}(x_0)\leq 1$ arbitrarily.
By Claims~\ref{claim:degree 0 vertex} and \ref{claim:degree 1 vertex}, there is a vertex $y_0\in N_G(x_0)\cap V(F)$ such that $N_G(y_0^+)\cap V(H)\neq \emptyset$.
Let $z_0$ be a vertex in $N_G(y_0^+)\cap V(H)$.
For $i=1,2,\dots$, we sequentially define $(D_i,x_i,y_i,z_i)$ in the following procedure until
$z_i\in \bigcup_{j=0}^{i}V(D_j)$.

Suppose that $(D_j,x_j,y_j,z_j)$ is defined for each $0\leq j\leq i-1$.
\begin{enumerate}
    \item Let $D_i$ be the component that contains $z_{i-1}$
    \item We define $x_i\in D_i$ and $y_i\in N_G(x_i)\cap V(F)$ as follows.
    \begin{enumerate}
        \item If $D_i$ is isomorphic to $K_1$, then let $x_i=z_{i-1}$. By Claim~\ref{claim:degree 0 vertex}, there is a vertex $y_i\in N_G(x_i)\cap (V(F)\setminus \{y_{i-1}^+\})$ such that $N_G(y_i^+)\cap V(H)\neq \emptyset$.
        \item If $D_i$ is not isomorphic to $K_1$, then let $x_i$ be a leaf of $D_i$ distinct from $z_{i-1}$.
        By Claim~\ref{claim:degree 1 vertex}, there is a vertex $y_i\in N_G(x_i)\cap V(F)$ such that $N_G(y_i^+)\cap V(H)\neq \emptyset$. 
        Note that it is possible that $y_i=y_{i-1}^+$ in this case.
    \end{enumerate}
    \item If there is a component $D_j\in \{D_0,\dots ,D_i\}$ such that $N_G(y_i^+)\cap V(D_j)\neq \emptyset$, then let $z_i\in N_G(y_i^+)\cap V(D_j)$ so that the index $j$ is as large as possible.
    Otherwise, let $z_i$ be an arbitrary vertex in $N_G(y_i^+)\cap V(H)$.
\end{enumerate}

Since the number of components of $H$ is finite, this procedure must end.
Without loss of generality, we may assume that the procedure stops at $(D_r,x_r,y_r,z_r)$ and $z_r\in V(D_0)$.
Furthermore, we may assume that $y_r^+$ is not adjacent to any components in $\{D_1,\dots , D_r\}$, and hence $y_r\notin \{y_1,\dots , y_{r-1}\}$.

By the choice of $r$, we know that $D_0,\dots , D_r$ are distinct components of $H$.
Also, for every $i$ and $j$ with $0\leq i<j\leq r-1$, if $y_j\in \{y_i^-,y_i\}$, then we can choose $D_i$ or $D_{i+1}$ as $D_{j+1}$, which contradicts the choice of $r$.
We consider the following two cases.

\medskip
\noindent
\textit{Case 1.} $y_r=y_0^-$.
\medskip

For each $i\in \{1,\dots ,r\}$, let $W_i$ be the $y_{i-1}^+ y_i$-walk $y_{i-1}^+ z_{i-1} P_i x_i y_i$, where $P_i$ is the unique $z_{i-1} x_i$-path in $D_i$.
If $D_i$ is isomorphic to $K_1$, then by 2(a), we chose $y_i$ so that $y_i\neq y_{i-1}^+$.
Thus, if $y_i=y_{i-1}^+$, then we know that $x_i\neq z_{i-1}$, and hence $W_i$ is a cycle of $G$.
Otherwise, $W_i$ is a $y_{i-1}^+ y_i$-path of $G$.

Since $D_1, \dots , D_r$ are pairwise distinct components of $H$, $V(W_i)\setminus \{y_{i-1}^+,y_i\}$ and $V(W_j)\setminus \{y_{j-1}^+,y_j\}$ are disjoint for different $i$ and $j$, and in particular, $W_1,\dots ,W_r$ are pairwise edge-disjoint.
Let $F_1$ be a graph obtained from $F$ by removing the edges $\{y_iy_i^+\mid 0\leq i\leq r\}$, adding the walks $W_1,\dots ,W_r$, and deleting $y_0$.
We can check that $F_1$ is a 2-regular subgraph of $G$ as follows.
It is easy to see that every vertex in $V(F_1)\setminus \bigcup_{i=0}^r\{y_i,y_i^+\}$ has degree 2 in $F_1$.
For each $i\in \{1,\dots ,r\}$, $y_i$ originally has degree two in $F$, loses degree one by deleting $y_iy_i^+$, and gains degree one by adding $W_i$, resulting in $d_{F_1}(y_i)=2+1-1=2$.
Similarly, for each $i\in \{0,\dots ,r-1\}$, $y_i^+$ originally has degree two in $F$, loses degree one by deleting $y_iy_i^+$, and gains degree one by adding $W_{i+1}$, and hence $d_{F_1}(y_i^+)=2+1-1=2$.
Note that when $y_{i-1}^+=y_i$, $y_{i-1}^+$ loses degree two by deleting $\{y_{i-1}y_{i-1}^+,y_iy_i^+\}$ and gains degree two by adding $W_i$.
Combining these, we conclude that $F_1$ is 2-regular.

If $r\geq 2$, then we have 
\[
|V(F_1)|=|V(F)\setminus \{y_0\}|+\sum_{i=1}^r |V(W_i)\setminus \{y_{i-1}^+,y_i\}|\geq |V(F)|-1+r>|V(F)|,
\]
a contradiction by the maximality of $|V(F)|$.
Similarly, if $r=1$ and $D_1$ is not isomorphic to $K_1$, the choice of $x_1$ implies that $|V(W_1)|\geq 4$, and hence 
\[|V(F_1)| \geq |V(F)|-1+|V(W_1)|-2 > |V(F)|,\]
a contradiction again.
Thus, we conclude that $r=1$ and $D_1$ is isomorphic to $K_1$, which implies that $V(F_1)=(V(F)\setminus \{y_0\})\cup \{z_0\}$.
Then, since $y_0$ is adjacent to a component $D_0$ of $H$ and $z_0$ is an isolated vertex of $H$, the number of isolated vertices of $G-V(F_1)$ is strictly less than that of $H=G-V(F)$, a contradiction.

\medskip
\noindent
\textit{Case 2.} $y_r\neq y_0^-$.
\medskip

We define $W_1,\dots , W_r$ similarly to Case 1.
Also, since $y_r^+\neq y_0$, let $W_0$ be a $y_r^+ y_0$-walk $y_r^+ z_r P_0 x_0 y_0$ of $G$, where $P_0$ is the unique $z_r x_0$-path in $D_0$.

Using an argument similar to that in the previous case, we infer that $V(W_i)\setminus \{y_{i-1}^+,y_i\}$ and $V(W_j)\setminus \{y_{j-1}^+,y_j\}$ are disjoint for any $0\leq i<j\leq r$, and that $W_0, \dots ,W_r$ are pairwise edge-disjoint. 
Let $F_2$ be a subgraph of $G$ obtained from $F$ by removing the edges $\{y_iy_i^+\mid 0\leq i\leq r\}$ and adding the walks $W_0,\dots ,W_r$.

Then, we can check that $F_2$ is a 2-regular subgraph of $G$ with 
\[
|V(F_2)|=|V(F)|+\sum_{i=1}^r |V(W_i)\setminus \{y_{i-1}^+,y_i\}|\geq |V(F)|+r>|V(F)|,
\]
a contradiction by the maximality of $|V(F)|$.
This completes the proof of Theorem~\ref{thm:main}.

\section{Remarks on algorithmic aspect of pseudo 2-factor}

Our proof gives an algorithm to find a pseudo 2-factor with at most $f(G)$ non-cycle components, but not a pseudo 2-factor with minimum number of non-cycle components.
It is known that there is a polynomial-time algorithm to give a maximum 2-regular subgraph of a given graph.
We remark that, for a given graph $G$, a pseudo 2-factor with minimum number of non-cycle components does not always contain a maximum 2-regular subgraph.
For instance, a graph in Figure~\ref{fig:algorithm} of 22 vertices has a maximum 2-regular subgraph of order 19, but every pseudo 2-factor with minimum number of non-cycle components contains a 2-factor with 18 vertices.
Thus, the following question remains open.

\begin{question}
    Is there a polynomial-time algorithm to find a pseudo 2-factor with minimum number of non-cycle components?
\end{question}

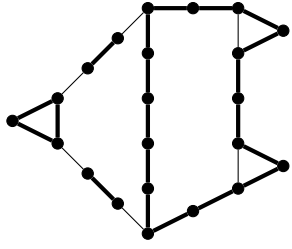
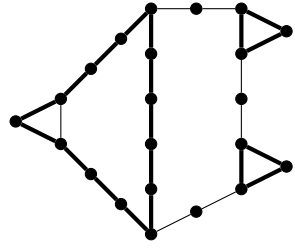
\begin{figure}
\centering
\begin{minipage}{0.45\columnwidth}
    \centering
    \begin{tikzpicture}[roundnode/.style={circle, draw=black,fill=black, minimum size=1.5mm, inner sep=0pt}]
        \node [roundnode] (v1) at (0,3){};
        \node [roundnode] (v2) at (0,2.4){};
        \node [roundnode] (v3) at (0,1.8){};
        \node [roundnode] (v4) at (0,1.2){};
        \node [roundnode] (v5) at (0,0.6){};
        \node [roundnode] (v6) at (0,0){};
        
        \node [roundnode] (u1) at (-0.4,2.6){};
        \node [roundnode] (u2) at (-0.8,2.2){};
        \node [roundnode] (u3) at (-1.2,1.8){};
        \node [roundnode] (u4) at (-1.8,1.5){};
        \node [roundnode] (u5) at (-1.2,1.2){};
        \node [roundnode] (u6) at (-0.8,0.8){};
        \node [roundnode] (u7) at (-0.4,0.4){};

        \node [roundnode] (w1) at (0.6,3){};
        \node [roundnode] (w2) at (1.2,3){};
        \node [roundnode] (w3) at (1.8,2.7){};
        \node [roundnode] (w4) at (1.2,2.4){};
        \node [roundnode] (w5) at (1.2,1.8){};
        \node [roundnode] (w6) at (1.2,1.2){};
        \node [roundnode] (w7) at (1.8,0.9){};
        \node [roundnode] (w8) at (1.2,0.6){};
        \node [roundnode] (w9) at (0.6,0.3){};

        \draw [ultra thick] (v1)--(v2)--(v3)--(v4)--(v5)--(v6)--(w9)--(w8)--(w7)--(w6)--(w5)--(w4)--(w3)--(w2)--(w1)--(v1);
        \draw [ultra thick] (u1)--(u2);
        \draw [ultra thick] (u3)--(u4)--(u5)--(u3);
        \draw [ultra thick] (u6)--(u7);

        \draw (v1)--(u1);
        \draw (u2)--(u3);
        \draw (u5)--(u6);
        \draw (u7)--(v6);
        \draw (w2)--(w4);
        \draw (w6)--(w8);
    \end{tikzpicture}
    \subcaption{A pseudo 2-factor with minimum number of non-cycle components.}
    \label{fig:pseudo 2factor}
\end{minipage}
\begin{minipage}{0.45\columnwidth}
    \centering
    \begin{tikzpicture}[roundnode/.style={circle, draw=black,fill=black, minimum size=1.5mm, inner sep=0pt}]
        \node [roundnode] (v1) at (0,3){};
        \node [roundnode] (v2) at (0,2.4){};
        \node [roundnode] (v3) at (0,1.8){};
        \node [roundnode] (v4) at (0,1.2){};
        \node [roundnode] (v5) at (0,0.6){};
        \node [roundnode] (v6) at (0,0){};
        
        \node [roundnode] (u1) at (-0.4,2.6){};
        \node [roundnode] (u2) at (-0.8,2.2){};
        \node [roundnode] (u3) at (-1.2,1.8){};
        \node [roundnode] (u4) at (-1.8,1.5){};
        \node [roundnode] (u5) at (-1.2,1.2){};
        \node [roundnode] (u6) at (-0.8,0.8){};
        \node [roundnode] (u7) at (-0.4,0.4){};

        \node [roundnode] (w1) at (0.6,3){};
        \node [roundnode] (w2) at (1.2,3){};
        \node [roundnode] (w3) at (1.8,2.7){};
        \node [roundnode] (w4) at (1.2,2.4){};
        \node [roundnode] (w5) at (1.2,1.8){};
        \node [roundnode] (w6) at (1.2,1.2){};
        \node [roundnode] (w7) at (1.8,0.9){};
        \node [roundnode] (w8) at (1.2,0.6){};
        \node [roundnode] (w9) at (0.6,0.3){};

        \draw [ultra thick] (v1)--(v2)--(v3)--(v4)--(v5)--(v6)--(u7)--(u6)--(u5)--(u4)--(u3)--(u2)--(u1)--(v1);
        \draw [ultra thick] (w2)--(w3)--(w4)--(w2);
        \draw [ultra thick] (w6)--(w7)--(w8)--(w6);

        \draw (v1)--(w1)--(w2);
        \draw (w4)--(w5)--(w6);
        \draw (w8)--(w9)--(v6);
        \draw (u3)--(u5);
    \end{tikzpicture}
    \subcaption{A maximum 2-regular subgraph.}
    \label{fig:max 2factor}
\end{minipage}
\caption{An example of a graph in which every pseudo 2-factor with minimum number of non-cycle components does not contain maximum 2-regular subgraphs.}
\label{fig:algorithm}
\end{figure}

\section*{Acknowledgement}
The author thanks Professor Katsuhiro Ota for giving me helpful comments.
The author is supported by JSPS KAKENHI grant number 25KJ2077.

\end{document}